\documentclass[12pt,reqno]{amsart}

\usepackage{amssymb}
\usepackage{amsmath}
\usepackage{mathpazo}
\usepackage{calc}
\usepackage{pstricks}

\theoremstyle{plain}
\newtheorem{theorem}{Theorem}[section]
\newtheorem*{Htheorem}{Howson's Theorem}

\newtheorem{lma}[theorem]{Lemma}
\newtheorem{cor}[theorem]{Corollary}

\theoremstyle{definition}
\newtheorem{expl}[theorem]{Example}
\newtheorem{rmk}[theorem]{Remark}

\newcommand{\sdp}{\ast_{\theta}}
\newcommand{\aut}[1]{\text{Aut}(#1)}
\newcommand{\A}{\mathcal{A}}
\newcommand{\B}{\mathcal{B}}
\newcommand{\mapright}[1]{\stackrel{#1}{\longrightarrow}}
\newcommand{\gen}[1]{\left\langle #1\right\rangle}
\newcommand{\inv}{^{-1}}
\newcommand{\discup}{\,\dot\cup\,}
\newcommand{\Z}{\mathbb{Z}}

\newcommand{\rk}[1]{\text{rk}(#1)}

\begin{document}

\setlength{\parindent}{0pt} \thispagestyle{empty}


\begin{flushleft}

\Large

\textbf{Howson's property for semidirect products of\\ semilattices by groups}

\end{flushleft}


\normalsize

\vspace{-2.25mm}

\rule{\textwidth}{1.5pt}

\vspace{2mm}


\begin{flushright}

\large

\textsc{Pedro V. Silva}

\normalsize

\vspace{3pt}


\scriptsize

\textit{Centro de Matem\'{a}tica, Faculdade de Ci\^{e}ncias, Universidade do Porto,}

\textit{Rua Campo Alegre 687, 4169-007 Porto, Portugal}

\texttt{pvsilva@fc.up.pt}

\vspace{2mm}

\large

\textsc{Filipa Soares}

\normalsize

\vspace{3pt}


\scriptsize

\textit{\'Area Departamental de Matem\'atica, ISEL,}

\textit{Rua Conselheiro Em\'idio Navarro 1, 1949-014 Lisboa, Portugal}

\& \hspace{2mm} \textit{Centro de \'Algebra da Universidade de Lisboa,}

\textit{Avenida Prof. Gama Pinto 2, 1649-003 Lisboa, Portugal}

\texttt{falmeida@adm.isel.pt}

\vspace{5mm}

\normalsize

\today

\end{flushright}

\medskip

\begin{flushleft}

\small

{\bf Abstract.} An inverse semigroup $S$ is a Howson inverse semigroup if
  the intersection of finitely generated inverse subsemigroups of $S$
  is finitely generated. Given a locally finite action $\theta$ of a
  group $G$ on a semilattice $E$, it is proved that $E \sdp G$ is a Howson
  inverse semigroup if and only if $G$ is a Howson
  group. It is also shown that this equivalence fails for arbitrary actions.

\medskip

{\bf Keywords.} $E$-unitary inverse semigroup,
 Howson's theorem, locally finite action,
 semidirect product of a semilattice by a group.

\medskip

{\bf 2010 Mathematics Subject Classification.} 20M18

\medskip

\date{\today}

\end{flushleft}

\bigskip

\section{Introduction}

In \cite{H54}, Howson proved a result that would become known as Howson's Theorem:

\begin{Htheorem}
 The intersection of two finitely generated subgroups of a free group
 is a finitely generated subgroup.
\end{Htheorem}

This property, not being true in general, led to defining a group $G$ to be
a \emph{Howson group} if the intersection of any two finitely
generated subgroups of $G$ 
is again a finitely generated subgroup of $G$.

Similarly, we say that an inverse semigroup $S$ is a \emph{Howson inverse semigroup} 
if the intersection of any two finitely generated inverse subsemigroups of $S$
is finitely generated. Note that if $S$ is a group, then the inverse
subsemigroups of $S$ are precisely its subgroups, thus $S$ is a Howson
inverse semigroup if and only if it is a Howson group.

We remark also that $S$ being a Howson inverse semigroup does not
imply that the intersection of finitely generated {\em subsemigroups}
of $S$ is finitely generated (i.e., $S$ needs not be a {\em Howson
semigroup}). A counterexample is provided by the free 
group $F$ of rank 2 \cite[Proposition 2.1(ii)]{BGS14}. 

Contrary to the behaviour of free groups, Jones and Trotter showed that,
although the free monogenic inverse semigroup is a Howson inverse
semigroup \cite[Theorem~1.6]{JT89}, that is not the case for any other
free inverse 
semigroup~\cite[Corollary~2.2]{JT89}. However, the
intersection of any two monogenic inverse subsemigroups of a free
inverse semigroup is always finitely generated \cite{Sil91}.

As one would expect, no general characterizations of Howson groups
are known. In what is probably the most general result of that kind,
Ara\'ujo, Sykiotis and the first author proved that every 
fundamental group of a finite graph of groups with
virtually polycyclic vertex groups and finite edge groups is a Howson
group \cite[Theorem 3.10]{ASS14}.

Thus the problem of identifying Howson inverse semigroups promises to
be even harder. Given the extraordinary importance assumed by
$E$-unitary inverse semigroups in the theory of inverse semigroups,
they constitute a good starting point, particularly the case of
semidirect products of semilattices by groups.

Indeed, O'Carroll proved in \cite{OC76} that every $E$-unitary inverse
semigroup $S$ embeds into some semidirect product $E \sdp G$ of a
semilattice by a group. The second author proved in \cite{Sil93} that
this embedding can be assumed to be normal-convex, i.e. every quotient
of $S$ embeds in some quotient of $E \sdp G$. Therefore it is a
natural problem to determine under which conditions a semidirect
product of a semilattice by a group is a Howson inverse semigroup. If the action
of $G$ on $E$ has a fixed point (i.e. if $G\cdot e = \{ e \}$ for some $e
\in E$) then $G$ embeds in $E \sdp G$ and so $G$ being a Howson group
is a necessary condition (cf. Lemma~\ref{isig}).
We note that if $E$ has an identity (maximum) or a
zero (minimum) then such an element is necessarily a fixed point for any
action of a group.

In Section 3, we show that if $E$ is a finite semilattice, then $E\sdp
G$ is a Howson inverse semigroup if and only if $G$ is a Howson
group. 
We also prove a theorem on polynomial bounds, introducing the
concept of a polynomially Howson inverse semigroup.
The main theorem of Section 3 is extended in Section 4 to arbitrary
  semilattices, provided that the group action is locally
  finite. Finally, in Section 
5, examples are produced to show that anything can happen when the
action is not locally finite.

\section{Preliminaries}

Let $E$ be a ($\wedge$-)semilattice and $G$ a group acting on the
left on $E$ via 
the homomorphism $\theta\colon G\to \aut{E}$.
As usual, we write $\theta_g$ instead of $\theta(g)$ and $g\cdot e$
instead of $\theta_g(e)$, 
for any $g\in G$ and $e\in E$.
In particular, $\theta$ being a homomorphism is equivalent to
$\theta_{gh}(e) = \theta_g (\theta_h(e))$, 
for any $g,h \in G$ and $e\in E$, that is, to $(gh)\cdot e=g\cdot(h\cdot e)$.

The action $\theta$ determines the {\em semidirect product} $E\sdp G$, where 
$$(e,g)(f,h)=(e\wedge (g\cdot f),gh),$$ 
for all $(e,g),(f,h)\in E\times G$.
Also recall that $(e,g)^{-1}=(g^{-1}\cdot e,g^{-1})$, for each
$(e,g)$. If the action is trivial, i.e. $\theta_g = id_E$ for every $g
\in G$, then we have the direct product $E \times G$.

Let $\sigma\colon E\sdp G \to E$ and $\gamma\colon E\sdp G \to G$
denote the projections ($\sigma$ of $\sigma$emilattice, $\gamma$ of
$\gamma$roup!); thus, $u=(\sigma(u),\gamma(u))$ whenever $u\in E\sdp G$.
Note that, except when $\theta$ is trivial, only $\gamma$ is a
homomorphism.

For further details on inverse semigroups, the reader is referred to
\cite[Chapter 5]{How95} and \cite{Law99}.

Given an inverse semigroup $S$ and a subset $X\subseteq S$,
we denote by $\gen{X}$ the inverse subsemigroup of $S$
generated by $X$. In particular, if $S$ is a group (respectively
semilattice), $\gen{X}$ is the subgroup (respectively
subsemilattice) of $S$ generated by $X$.
For a finitely generated inverse semigroup $S$,
the \emph{rank} of $S$ is defined as
$$\rk{S} = \min \{ |X| \colon S=\gen{X} \} \, . $$

If $A$ is a finite nonempty alphabet, a finite $A$-{\em automaton} is
a quadruple of 
the form ${\A} = (Q,q_0,T,\Gamma)$, where $Q$ is a finite set (vertices), $q_0
\in Q$, $T \subseteq Q$ and $\Gamma \subseteq Q \times A \times Q$
(edges). A {\em path} in $\A$ is a sequence of the form
\begin{equation}
\label{path}
p_0 \mapright{a_1} p_1 \mapright{a_2} \ldots \mapright{a_n} p_n,
\end{equation}
with $n \geq 1$ and $(p_{i-1},a_i,p_i) \in E$ for $i =
1,\ldots,n$. Note that we are not admitting empty paths in this paper!
The path (\ref{path}) has length $n \geq 1$ and label $a_1a_2\ldots
a_n \in A^+$. It is {\em successful} if $p_0 = q_0$ and $p_n \in
T$. The {\em language} of $\A$, denoted by $L(\A)$, is the
subset of $A^+$ consisting of the labels of all successful paths
in $\A$. A language $L \subseteq A^+$ is {\em rational} if $L =
L(\A)$ for some finite $A$-automaton $\A$.

A finite $A$-automaton ${\A} = (Q,q_0,T,\Gamma)$ is:
\begin{itemize}
\item
{\em deterministic} if $(p,a,q),(p,a,q') \in E$ implies $q = q'$ for
all $p,q,q' \in Q$ and $a \in A$;
\item {\em complete} if for all $p \in Q$ and $a \in A$ there exists
  some $q \in Q$ such that $(p,a,q) \in E$.
\end{itemize}

Let $S$ be a semigroup. We say that $X \subseteq S$ is a {\em rational
  subset} of $S$ if there exists a finite alphabet $A$, a homomorphism
$\varphi\colon A^+ \to S$ and a rational language $L \subseteq A^+$ such
that $X = \varphi(L)$. The following result, proved by Anisimov and
Seifert in \cite{AS75}, will be important for us:

\begin{theorem} \label{anisei}
 Let $G$ be a group and let $H$ be a subgroup of $G$.
 Then $H$ is a rational subset of $G$ if and only if $H$ is finitely generated.
\end{theorem}

For more details on languages, automata and rational subsets, the
reader is referred to \cite{B79}.

\section{Finite semilattices}

In this section we consider a group $G$ acting on the left on a finite
semilattice $E$.

We start with a useful lemma, proved with some help from automata
theory.

\begin{lma}
\label{newle}
Let $E$ be a finite semilattice and $G$ a group acting on the left on
 $E$ via the homomorphism 
 $\theta\colon G\to {\rm Aut}(E)$.
Let $X$ be a finite nonempty subset of $E\sdp G$ and let $e \in E$.
Let
$$S(e) = \{ u \in \gen{X} \colon \sigma(u) \geq
e,\;\theta_{\gamma(u)} = id_E \}.$$
Then $\gamma(S(e))$ is either empty or a finitely generated subgroup of $G$.
\end{lma}

\begin{proof}
For simplicity, write $S$ instead of $S(e)$.
Assume that $S \neq \varnothing$. We show first that $S$ is an inverse
subsemigroup of $E \sdp G$.

Suppose that $u,v \in S$. Then
$uv \in \gen{X}$. Moreover,
$$\begin{array}{lll}
(\sigma(uv),\gamma(uv))&=&uv =
(\sigma(u),\gamma(u))(\sigma(v),\gamma(v))\\
&=&(\sigma(u) \wedge
(\gamma(u) \cdot \sigma(v)), \gamma(u)\gamma(v)).
\end{array}$$
Now $\gamma(u) \cdot \sigma(v) = \theta_{\gamma(u)}(\sigma(v)) =
\sigma(v)$, hence
$$\sigma(uv) = (\sigma(u) \wedge \sigma(v)) \geq (e \wedge e) = e$$
and also
$\theta_{\gamma(uv)} = \theta_{\gamma(u)}\theta_{\gamma(v)} =
id_E$. Thus $uv \in S$.

On the other hand, we have also $u\inv \in \gen{X}$ and
$$\begin{array}{lll}
(\sigma(u\inv),\gamma(u\inv))&=&u\inv = ((\gamma(u))\inv \cdot
\sigma(u),(\gamma(u))\inv)\\
&=&(\theta_{\gamma(u)}\inv(\sigma(u)),(\gamma(u))\inv) =
(\sigma(u),(\gamma(u))\inv). 
\end{array}$$
Since $\sigma(u) \geq e$ and $\theta_{(\gamma(u))\inv} =
\theta_{\gamma(u)}\inv = id_E$, we get $u\inv \in S$ and so $S$ is an inverse
subsemigroup of $E \sdp G$.

We show next that $S$ is a rational subset of $E \sdp G$. Assuming
that $X = X\inv$, we introduce
a finite alphabet $A = \{ a_x \colon x \in X\}$. Let $\varphi \colon A^+ \to
E \sdp G$ be the homomorphism defined by $\varphi(a_x) = x$. We define
an $A$-automaton ${\A} = (Q,q_0,T,\Gamma)$ by
$$\begin{array}{lll}
Q&=&\{ q_0\} \discup (E \times \aut{E}),\\
T&=&\{ f \in E \colon f \geq e\} \times \{ id_E \},\\
\Gamma&=&\{ (q_0,a_x,(\sigma(x),\theta_{\gamma(x)}))\colon x \in X\}\\
&\cup&\{ ((f,\pi),a_x,(f\wedge \pi(\sigma(x)),\pi\theta_{\gamma(x)}))
\colon f \in E,\; \pi \in \aut{E},\; x \in X \}.
\end{array}$$ 

It follows from the definition that $\A$ is complete and
deterministic. Thus, for every $v \in A^+$, there exists a unique $q_v
\in Q$ such that $q_0 \mapright{v} q_v$ is a path in $\A$. We
show that
\begin{equation}
\label{newle1}
q_v = (\sigma(\varphi(v)),\theta_{\gamma(\varphi(v))}).
\end{equation}

We use induction on $|v|$. The case $|v| = 1$ being immediate, we
assume that $|v| > 1$ and the claim holds for shorter words. Then we
may write $v = wa_x$ with $x \in X$. By the induction hypothesis, we
have a path
$$q_0 \mapright{w} (\sigma(\varphi(w)),\theta_{\gamma(\varphi(w))}),$$
and also an edge
$$(\sigma(\varphi(w)),\theta_{\gamma(\varphi(w))}) \mapright{a_x}
(\sigma(\varphi(w)) \wedge
\theta_{\gamma(\varphi(w))}(\sigma(x)),
\theta_{\gamma(\varphi(w))}\theta_{\gamma(x)}).$$    
Since
$\varphi(v) = \varphi(w)x =
(\sigma(\varphi(w)),\gamma(\varphi(w)))(\sigma(x),\gamma(x))$ yields
$$\sigma(\varphi(v)) = (\sigma(\varphi(w)) \wedge 
\theta_{\gamma(\varphi(w))}(\sigma(x)))$$
and $\theta_{\gamma(\varphi(v))} =
\theta_{\gamma(\varphi(w))}\theta_{\gamma(x)}$, there exists a path
$$q_0 \mapright{v} (\sigma(\varphi(v)),\theta_{\gamma(\varphi(v))}),$$
in $\A$. Hence (\ref{newle1}) holds for $v$ and therefore for any
word of $A^+$.

We show next that
\begin{equation}
\label{newle2}
\varphi(L({\A})) = S.
\end{equation}

Let $v \in L({\A})$. Since $\varphi(A) = X = X \cup 
X\inv$, we have $\varphi(L({\A})) \subseteq \varphi(A^+) =
\gen{X}$. It follows from the uniqueness of $q_v$ that $q_v \in T$,
hence (\ref{newle1}) yields
$\sigma(\varphi(v)) \geq e$ and $\theta_{\gamma(\varphi(v))} = id_E$. 
Thus $\varphi(v) \in S$ and so $\varphi(L({\A})) \subseteq S$.

Suppose now that $u \in S$. Since $S \subseteq \gen{X}$, we may write
$u = x_1\ldots x_n$ for some $x_i \in X$. Let $v = a_{x_1}\ldots
a_{x_n}$. Since $u \in S$, we have $\sigma(\varphi(v)) = \sigma(u) \geq
e$ and $\theta_{\gamma(\varphi(v))} = \theta_{\gamma(u)} = id_E$. In
view of (\ref{newle1}), we get $q_v \in T$, hence $v \in L({\A})$
and so $u \in \varphi(L({\A}))$. Therefore $S \subseteq
\varphi(L({\A}))$ and so (\ref{newle2}) holds.

It follows that $S$ is a rational inverse subsemigroup of $E \sdp
G$. Composing $\varphi$ with $\gamma$, we deduce that $\gamma(S)$ is a
rational inverse subsemigroup of $G$, i.e. a rational subgroup of $G$. By
Theorem \ref{anisei}, $\gamma(S)$ is finitely generated.
\end{proof}


Before we proceed, we make the following observation.

\begin{rmk} \label{R:rank-of-Se}
 The very proof of Anisimov and Seifert's Theorem
 (or rather, a few proofs among many)
 allows us to draw some conclusions on the rank of $\gamma(S(e))$.
 
 Indeed, if a subgroup $H$ of a group $G$ is rational, we may write
 $H=\rho(L(\B))$ for some (finite) $(B\cup B^{-1})$-automaton $\B$
 and homomorphism $\rho\colon (B \cup B^{-1})^{\ast}\to G$ satisfying
 $\rho(b\inv) = (\rho(b))\inv$ for every $b \in B$. Then $H$ is
 generated by $\rho(Y)$ where $Y$ consists of all words in
 $\rho^{-1}(H)$ 
 of length less than twice the number of vertices of $\B$
 (cf.~\cite[Theorem~3.1]{BS10}). As
 $|Y|\leq \sum_{i=0}^{2m-1} (2|B|)^i,$
 where $m$ denotes the number of vertices of $\B$,
 we have that
 $$\rk{H}\leq |\rho(Y)|\leq \sum_{i=0}^{2m-1} (2|B|)^i =
 \frac{1-(2|B|)^{2m}}{1-2|B|} \, .$$ 

Assume that $S(e) \neq \emptyset$. As $\gamma(S(e))=(\gamma\varphi)(L(\A))$,
 the automaton $\A$ has $$|E|\times |\aut{E}|\leq |E|\times (|E|-1)!=|E|!$$ vertices,
 and its alphabet of labels has size $|X|$, we deduce that
 \begin{equation}
\label{rkbou}
\rk{\gamma(S(e))}\leq \sum_{i=0}^{2|E|!-1} (2|X|)^i = \frac{1-(2|X|)^{2\times|E|!}}{1-2|X|}.
\end{equation}
\end{rmk}

We continue with another lemma.

\begin{lma}
\label{isig}
 Let $G$ be a group acting on the left on a semilattice $E$ via the
 homomorphism $\theta\colon G\to {\rm Aut}(E)$. If $\theta$ has a fixed
 point and $E\sdp G$ is a Howson inverse semigroup, then $G$ is a 
 Howson group. 
\end{lma}

\begin{proof}
Let $e \in E$ be such that $g\cdot e = e$ for every $g \in G$. Then $G$ is
isomorphic to the inverse subsemigroup $\{ e \} \times G$ of $E \sdp
G$. Since $E\sdp G$ is a Howson inverse semigroup, so is $G$. Since
the inverse subsemigroups of a group are its subgroups, $G$ is a
Howson group.
\end{proof}

Now we can prove the main result of this section.

\begin{theorem} \label{T:howson}
 Let $E$ be a finite semilattice and $G$ a group acting on the left on
 $E$ via the homomorphism 
 $\theta\colon G\to {\rm Aut}(E)$. Then the following conditions are
 equivalent:
\begin{itemize}
\item[(i)] $E \sdp G$ is a Howson inverse semigroup;
\item[(ii)] $G$ is a Howson group.
\end{itemize}
\end{theorem}

\begin{proof}
(i) $\Rightarrow$ (ii). Since $E$ is finite, it has a zero
  $0$, which is necessarily a fixed point of $\theta$. Thus we may
  apply Lemma \ref{isig}.

(ii) $\Rightarrow$ (i). Let $X_1,X_2 \subseteq E\sdp G$ be finite. 
We build a finite generating set $X$ for $\gen{X_1}\cap \gen{X_2}$ as follows.

For $i = 1,2$, let 
$$P_i = \{ (\sigma(u), \theta_{\gamma(u)}) \colon u \in \gen{X_i} \}
\subseteq E \times \aut{E}.$$
For all $(e,\pi) \in P_1 \cap P_2$ and $i = 1,2$, let
$$\begin{array}{lll}
S_i(e,\pi)&=&\{ u \in \gen{X_i} \colon \sigma(u) \geq
\pi\inv(e),\;\theta_{\gamma(u)} = id_E \},\\ 
S'_i(e,\pi)&=&\{ u \in \gen{X_i} \colon \sigma(u) =
e,\;\theta_{\gamma(u)} = \pi \},\\
S'(e,\pi)&=&S'_1(e,\pi) \cap S'_2(e,\pi),\\
H(e,\pi)&=&\gamma(S_1(e,\pi)) \cap \gamma(S_2(e,\pi)).
\end{array}$$
Write also
$$\begin{array}{l}
P = \{ (e,\pi) \in P_1 \cap P_2 \colon S'(e,\pi) \neq \emptyset \},\\
Q = \{ (e,\pi) \in P_1 \cap P_2 \colon H(e,\pi) \neq \emptyset \}.
\end{array}$$
By Lemma \ref{newle}, $\gamma(S_i(e,\pi))$ is a finitely generated
subgroup of $G$ for all $(e,\pi) \in Q$ and $i = 1,2$.

Now, since $G$ is a Howson group, then $H(e,\pi)$ is also a finitely
generated subgroup of $G$ for 
every $(e,\pi) \in Q$. Let $Y(e,\pi)$ denote a finite generating set
(closed under inversion) of $H(e,\pi)$, and
let $X(e,\pi)=\{\pi^{-1}(e)\}\times Y(e,\pi)$.


Finally, for every $(e,\pi) \in P$, we fix some element
$w_{(e,\pi)}\in S'(e,\pi)$. We claim that the finite set
\begin{equation} \label{gen-set}
 X = \Big( \bigcup_{(e,\pi)\in Q} X(e,\pi)\Big) \cup
    \{ w_{(e,\pi)} \colon (e,\pi)\in P \}
\end{equation}
generates $\gen{X_1}\cap \gen{X_2}$.
 
Let $(e,\pi) \in Q$ and $(\pi\inv(e),g) \in X(e,\pi)$. Let $i \in \{ 1,2\}$.
Since $(e,\pi) \in Q \subseteq P_i$, we have $(e,\pi) =
 (\sigma(u_i), \theta_{\gamma(u_i)})$ for some $u_i \in
\gen{X_i}$. Writing $h_i = \gamma(u_i)$, we have $(e,h_i) \in
\gen{X_i}$, hence $(h_i\inv\cdot e,h_i\inv) = (e,h_i)\inv \in
\gen{X_i}$. Since
$$\pi\inv(e) = \theta_{\gamma(u_i)}\inv(e) = \theta_{h_i}\inv(e) =
\theta_{h_i\inv}(e) = h_i\inv\cdot e,$$
we get
\begin{equation}
\label{np2}
(\pi\inv(e),1) = (h_i\inv\cdot e,1) = (h_i\inv\cdot
e,h_i\inv)(e,h_i) \in \gen{X_i}.
\end{equation}

Now $(\pi\inv(e),g) \in X(e,\pi)$ yields $g \in
\gamma(S_i(e,\pi))$, hence there exists some $v_i \in S_i(e,\pi)
\subseteq \gen{X_i}$ such
that $g = \gamma(v_i)$. Moreover,
$\sigma(v_i) \geq \pi\inv(e)$.
Writing $f = \sigma(v_i)$, it follows that
$(f,g) = (\sigma(v_i),\gamma(v_i)) = v_i \in \gen{X_i}$.  
Since $f = \sigma(v_i) \geq
\pi\inv(e)$, and in view of (\ref{np2}), we get
$$(\pi\inv(e),g) = (\pi\inv(e),1)(f,g) \in \gen{X_i}.$$
Since $i \in \{ 1,2\}$ is arbitrary, it follows that $X(e,\pi)
\subseteq \gen{X_1}\cap \gen{X_2}$ for every $(e,\pi) \in Q$.

On the other hand, it follows from the definitions that $S'(e,\pi)
\subseteq \gen{X_1}\cap \gen{X_2}$, hence $w_{(e,\pi)} \in
\gen{X_1}\cap \gen{X_2}$ for every $(e,\pi) \in P$. Therefore $X
\subseteq \gen{X_1}\cap \gen{X_2}$.

Conversely, suppose that $u\in \gen{X_1}\cap \gen{X_2}$. Write $e =
\sigma(u)$, $g = \gamma(u)$ and $\pi = \theta_g$. Then $u \in
S'(e,\pi)$, hence $(e,\pi) \in P$ and so $w = w_{(e,\pi)} \in X$. Thus
we may write $w = 
(e,h)$ for some $h \in G$ satisfying $\theta_h = \pi =
\theta_g$. Therefore $\theta_{h\inv g} = id_E$. 

On the other hand,
$$(\pi\inv(e),h\inv g) = (\pi\inv(e),h\inv)(e,g) = (e,h)\inv(e,g) =
w\inv u \in \gen{X_1} \cap \gen{X_2},$$ 
yields
$$(\pi\inv(e),h\inv g) \in S_1(e,\pi) \cap S_2(e,\pi).$$
Thus $h\inv g \in H(e,\pi)$ and so we may write $h\inv g = y_1\ldots
y_n$ for some $y_j \in Y(e,\pi)$, yielding $(\pi\inv(e),y_j) \in
X(e,\pi)$ for $j = 1,\ldots,n$. We show that
\begin{equation}
\label{np3}
(\pi\inv(e),h\inv g) = (\pi\inv(e),y_1)\ldots (\pi\inv(e),y_n).
\end{equation}
Indeed, $y_j \in Y(e,\pi) \subseteq \gamma(S_1(e,\pi)) \cap
\gamma(S_2(e,\pi))$ implies that $\theta_{y_j} = id_E$ for every $j$,
hence (\ref{np3}) follows from the decomposition $h\inv g = y_1\ldots
y_n$. 

Therefore 
$$u = (e,g) = (e,h)(h\inv \cdot e,h\inv g) = w(\pi\inv(e),h\inv g) \in
\gen{X}$$
and so $\gen{X} = \gen{X_1}\cap\gen{X_2}$ as claimed. Therefore
$\gen{X_1}\cap\gen{X_2}$ is finitely generated and so $E\sdp G$ is a
Howson inverse semigroup.
\end{proof}

\medskip

In his 1954 paper on the intersection of finitely generated free groups~\cite{H54},
Howson also provided an upper bound on the rank of $H_1\cap H_2$
in terms of the ranks of the (nontrivial) subgroups $H_1$ and $H_2$, namely:
\[
 \rk{H_1\cap H_2}\leq 2\,\rk{H_1}\rk{H_2}-\rk{H_1}-\rk{H_2}+1 \, .
\]
Two years later, Hanna Neumann~\cite{N56} improved this upper bound to
\[
 \rk{H_1\cap H_2}\leq 2(\rk{H_1}-1)(\rk{H_2}-1)+1
\]
and conjectured that the factor $2$ could in fact be removed from the
inequality, 
in what would become known as the ``Hanna Neumann Conjecture''.
In its full generality, the conjecture would only be proved in 2012,
independently by Friedman~\cite{F14} and Mineyev~\cite{M13}.

We say that an inverse semigroup $S$ is {\em polynomially Howson} if
  there exists a polynomial $p(x) \in \mathbb{R}[x]$ such that 
$$\rk{T_1},\rk{T_2} \leq n \Rightarrow \rk{T_1\cap T_2} \leq p(n)$$
for all inverse subsemigroups $T_1,T_2$ of $S$ and $n \in \mathbb{N}$.
If $p(x)$ can be taken to be quadratic, we say that $S$ is {\em
  quadratically Howson}. Note that by \cite{H54} free groups are
quadratically Howson.

\begin{theorem}
\label{polhow}
Let $E$ be a finite semilattice and $G$ a polynomially Howson group
acting on the left on $E$ via the homomorphism 
 $\theta\colon G\to {\rm Aut}(E)$. Then $E \sdp G$ is polynomially Howson.
\end{theorem}

\begin{proof}
Let $p(x) \in \mathbb{R}[x]$ be such that 
$$\rk{H_1},\rk{H_2} \leq n \Rightarrow \rk{H_1\cap H_2} \leq p(n)$$
for all subgroups $H_1,H_2$ of $G$ and $n \in \mathbb{N}$.

Let $X_1,X_2 \subseteq E\sdp G$ have at most $n$ elements. Write
$T_i=\gen{X_i}$ for $i = 1,2$. By the proof of Theorem \ref{T:howson},
 their intersection $T_1\cap T_2$ is generated by $X=Y\cup W$, where
$$Y=\bigcup_{(e,\pi)\in Q} X(e,\pi) \quad \text{ and } \quad
    W=\{ w_{(e,\pi)} \colon (e,\pi)\in P \}$$
(cf.~(\ref{gen-set})).
Since $P,Q\subseteq E\times \aut{E}$, we have $|P|,|Q|\leq |E|!$.
Thus, in particular, $|W|\leq |E|!$.
Moreover, for each $(e,\pi)\in Q$, we have
$|X(e,\pi)|=|Y(e,\pi)|=\rk{\gamma(S_1(e,\pi)) \cap \gamma(S_2(e,\pi))}$,
where, by Remark~\ref{R:rank-of-Se},
$$\rk{\gamma(S_i(e,\pi))}\leq q(2\rk{T_i}) \leq q(2n)$$
for $q(x) = \sum_{j=0}^{2\times|E|!-1} x^j$.

Thus,
$$\rk{\gamma(S_1(e,\pi)) \cap \gamma(S_2(e,\pi))} \leq
p(q(2n))$$
and so
$$\rk{T_1\cap T_2}\leq |E|!(1+p(q(2n))).$$
Therefore $S$ is polynomially Howson.
\end{proof}

\begin{cor}
\label{fgpolhow}
Let $E$ be a finite semilattice and $G$ a group
acting on the left on $E$ via the homomorphism 
 $\theta\colon G\to {\rm Aut}(E)$. If $G$ is the fundamental group of
a finite graph of groups with virtually polycyclic vertex groups
and finite edge groups, then $E \sdp G$ is polynomially Howson.
\end{cor}

\begin{proof}
By \cite[Theorem 3.10]{ASS14}, $G$ is quadratically Howson, hence we
may apply Theorem \ref{polhow}.
\end{proof}


\section{Locally finite actions}

Let $E$ be a semilattice and $G$ a group acting on $E$. Given a
subgroup $H \leq G$ and $e \in E$, we say that $H\cdot e = \{ h\cdot e
\colon h \in H\}$ is the $H$-{\em orbit} of $e$. The action of $G$ on
$E$ is said to be \emph{locally finite} if all the $H$-orbits are
finite whenever $H$ is a finitely generated subgroup of $G$.

We can extract from Theorem \ref{T:howson} the following corollary.

\begin{cor}
\label{lofi}
Let $E$ be a semilattice and $G$ a group acting on the left on
 $E$ by means of a locally finite action $\theta\colon G\to
{\rm Aut}(E)$. Then the following conditions are equivalent:
\begin{itemize}
\item[(i)] $E \sdp G$ is a Howson inverse semigroup;
\item[(ii)] $G$ is a Howson group.
\end{itemize}
\end{cor}

\begin{proof}
(i) $\Rightarrow$ (ii). Let $H,H'$ be finitely generated subgroups of
  $G$ and let $G' = \langle H \cup H'\rangle \leq G$. We fix some $e
  \in E$. Since $\theta$ is locally finite, $G'\cdot e$ is a finite subset of
  $E$. Let $E'$ denote the (finite)
  subsemilattice of $E$ generated by $G'\cdot e$. We claim that
  $G'\cdot E' \subseteq E'$. 

Indeed, the elements of $E'$ are of the form 
$(g'_1\cdot e) \wedge \ldots \wedge (g'_n\cdot e)$, 
with $g'_i \in G'$. If $g' \in G'$, then
$$g'\cdot ((g'_1\cdot e) \wedge \ldots \wedge (g'_n\cdot e)) =
(((g'g'_1)\cdot e) \wedge \ldots \wedge ((g'g'_n)\cdot e)) \in E'$$
and so $G'\cdot E' \subseteq E'$.

Let $\theta'\colon G' \to \aut{E'}$ be defined by $\theta'_{g'}
= \theta_{g'}|_{E'}$ for $g' \in G'$.  
Since $G'\cdot E' \subseteq E'$ and
$(\theta'_{g'})\inv = \theta'_{(g')\inv}$, it is easy to check that $\theta'$ is
a well-defined group homomorphism. Moreover, there is a natural
embedding of $E' \ast_{\theta'} G'$ into $E \sdp G$.
Thus $E' \ast_{\theta'} G'$ is a Howson inverse semigroup. 

Since $E'$ is finite, it has a zero
  $0$, which is necessarily a fixed point of $\theta'$. By Lemma
\ref{isig}, $G'$ is a Howson group. Since $H,H' \subseteq G'$, it
follows that $H\cap H'$ is finitely generated. Therefore $G$ is a
Howson group.

(ii) $\Rightarrow$ (i). Let $X_1,X_2$ be two finite nonempty subsets
of $E\sdp G$. We may 
assume that $X_i = X_i\inv$ for $i = 1,2$. Let 
$$Y = \{ \gamma(x) \colon x \in X_1 \cup X_2 \},\quad F = \{
\sigma(x) \colon x \in X_1 \cup X_2 \}.$$
Let $H$ be the subgroup of $G$ generated by $Y$. Since $G$ is a Howson
group, so is $H$. Let
$E'$ be the subsemilattice of $E$ generated by $H\cdot F = \{ h\cdot f \colon h
\in H,\; f \in F \}$.

Since $\theta$ is locally finite, $H\cdot F$ is a finite subset of
$E$. Since finitely generated semilattices are finite, it follows that
$E'$ is finite. An argument as the one above shows that $H\cdot E'
\subseteq E'$.  


Again as in the proof of the direct implication,
$\theta$ induces an 
action $\theta'\colon H \to \aut{E'}$ and we may view $E' \ast_{\theta'} H$
as an inverse subsemigroup of $E \sdp G$.  

Now we note that $X_i \subseteq F \times Y \subseteq E' \ast_{\theta'} H$ implies
$\gen{X_i} \subseteq E' \ast_{\theta'} H$ for $i = 1,2$. Since $E'$ is
finite and $H$ is Howson, we may use Theorem \ref{T:howson} to deduce
that $\gen{X_1} 
\cap \gen{X_2}$ is finitely generated. Therefore $E \sdp G$ is a
Howson inverse semigroup.
\end{proof}

We discuss now some examples.

\begin{expl}
A trivial example of a locally finite action is that of a trivial action,
 that is, an action in which $g\cdot e=e$, for all $e\in E$ and
 $g\in G$. Therefore the direct product of a semilattice by a Howson group is
 always a Howson inverse semigroup.
\end{expl}

If $G$ is a {\em locally finite} group (i.e. every finitely generated
subgroup of $G$ is finite), then $G$ is trivially a Howson group and the
action of $G$ on any semilattice 
is obviously locally finite. Thus we obtain the following consequence.

\begin{cor}
Let $G$ be a locally finite group acting on the left on a semilattice
$E$ via a homomorphism $\theta \colon G \to {\rm Aut}(E)$. Then $E \sdp G$ is
a Howson inverse semigroup. 
\end{cor}

The following example illustrates such a case. 

\begin{expl}
Let $FS_A$ be the free semilattice on a nonempty set $A$ and let
$S_A^f$ be the set of all permutations on $A$ with finite
support. Then $S_A^f$ is a locally finite group and so the natural action of
$S_A^f$ on $FS_A$ is locally finite.
\end{expl}

Indeed, $FS_A$ is the set of all finite nonempty subsets of $A$,
endowed with the union operation, and $S_A^f$ consists of all the
permutations of $A$ which fix all but finitely many elements of
$A$. Then $S_A^f$ acts on $FS_A$ by restriction, and is obviouslly
locally finite.

\medskip

Let $E$ be a semilattice with identity 1. We say that $E$ is {\em
  finite above} if $\{ f \in E \colon f \geq e\}$ is finite for every
$e \in E$. Given such a semilattice $E$, we define its {\em height
  function} as the function $\lambda:E \to \mathbb{N}$ defined by
$$\lambda(e) = \max\{ n \in \mathbb{N} \colon \mbox{ there exists a
  chain $1 = e_0 > \ldots > e_n = e$ in }E\}.$$
We say that $E$ is {\em strongly finite above} if it is finite above
and $\lambda\inv(n)$ is
finite for every $n \in \mathbb{N}$.

\begin{lma}
\label{lhf}
The action of any group on a strongly finite above
semilattice with identity is locally finite.
\end{lma}

\begin{proof}
Indeed, it is straightforward to check that $\lambda(g \cdot e) =
\lambda(e)$ for all $g \in G$ and $e \in E$. This can be achieved by
induction on $\lambda\inv(n)$, starting with $\lambda\inv(0) = \{ 1
\}$, where 1 denotes the identity of $E$. Therefore $G \cdot e
\subseteq \lambda\inv(\lambda(e))$ and the 
action is locally finite.
\end{proof}

We provide next an example of a nontrivial action
where:
\begin{itemize}
\item the semilattice is infinite, has an identity and is strongly
  finite above; 
\item
the group is Howson but not locally finite.
\end{itemize}

\begin{expl}
 Let
$$E = \{ (2n+1,0) \colon n \geq 0\} \cup (\bigcup_{n \geq 1} \{ 2n \}
\times \mathbb{Z}/n\mathbb{Z}),$$
partially ordered by
$$(k,x) \leq (\ell,y) \mbox{ if }k < \ell \mbox{ or } (k,x) =
(\ell,y).$$
Then $E$ is clearly infinite with identity $(0,1)$ and finite above, and the 
heigth function is given by $\lambda(k,x) = k-1$. Thus $E$ is strongly
finite above. Let $G$ be the additive group $\mathbb{Z}$, which is
Howson but not locally finite. We define a (nontrivial) action $\theta \colon G
\to \aut{E}$ 
by
$$\theta_m(2n+1,0) = (2n+1,0),\quad \theta_m(2n, k+n\mathbb{Z}) =
(2n,m+k+n\mathbb{Z}).$$
\end{expl}

It is straightforward to check that $\theta$ is well defined, hence
it is locally finite by Lemma \ref{lhf}. Therefore
$E\sdp G$ is a Howson inverse semigroup by Corollary \ref{lofi}.

\section{Non locally finite actions}

The next two examples show that in this case $E \sdp G$ may be a
Howson inverse semigroup or not.

\begin{expl}
By O'Carroll's construction~\cite{OC76}, any free inverse semigroup $F$
embeds in a semidirect product $S$ of a semilatttice by a free group. If
$F$ is not monogenic, then $F$ is not a Howson inverse semigroup by 
\cite[Corollary~2.2]{JT89}, therefore $S$ is not a Howson inverse
semigroup either. However, free groups are Howson
groups. In view of Corollary \ref{lofi}, the action cannot be locally
finite.
\end{expl}

The next example shows that the action being locally finite is not a
necessary condition for the semidirect product to be a Howson inverse
semigroup.
 
\begin{expl}
Let $G$ be $(\Z,+)$ and let $E$ be $\Z$ with the usual ordering. We
consider the action $\theta \colon G \to {\rm Aut}(E)$ defined by
$\theta_n(m) = n+m$. Then $\theta$ is not locally finite but $E\sdp G$
is a Howson inverse semigroup.
\end{expl}

It is straightforward to check that $\theta$ is well defined and is
not locally finite. Note also that $G$, being a free group, is a
Howson group.

We say that an inverse subsemigroup $S \leq E \sdp G$ is {\em bounded}
if there exists some $M \in \Z$ such that $m \leq M$ for every $(m,n)
\in S$. It is easy to see that:
\begin{itemize}
\item
if $A = \{ (m_1,n_1), \ldots, (m_k,n_k) \}$ is a finite subset of $E \sdp
G$ closed under inversion, then $\langle A \rangle$ is bounded
(with $M = \max\{ m_1,\ldots,m_k\}$);
\item
the intersection of bounded inverse subsemigroups of $E\sdp G$ is
bounded.
\end{itemize}

Now we show that
\begin{equation}
\label{zz}
\begin{array}{c}
\mbox{if $S \leq E\sdp G$ is bounded and contains a nonidempotent,}\\
\mbox{then $S$ is finitely generated.}
\end{array}
\end{equation}

Given that $S$ is by assumption inverse and contains a nonidempotent,
the positive integer
$$N = \min\{ n > 0\colon (m,n) \in S \mbox{ for some } m \in \Z \}$$
is well-defined.
It follows easily from the division algorithm that 
$$S \subseteq \Z \times N\Z$$
(notice that, if $n,p>0$, then $(k,n)^p=(k,np)$).
For $i = 0,\ldots,N-1$, let 
$$S_i = \{ (m,n) \in S \colon m \equiv i\,(\mbox{mod}\, N) \}.$$
Then each nonempty $S_i$ is an inverse subsemigroup of $S$. Therefore
it suffices to show that each nonempty $S_i$ is finitely generated. 

Fixing such an $i$, and since $S$ is bounded, we can define
$M_i = \max\{ m \colon (m,n) \in S_i,\; n > 0\}$. Let
$$S'_i = \{ (m,0) \in S_i\colon m > M_i\}.$$

We claim that
\begin{equation}
\label{zz1}
S_i = \langle \{ (M_i,N)\} \cup S'_i \rangle.
\end{equation}

We show first that $(M_i,N) \in S_i$. By definition of $M_i$, we have
$(M_i,n) \in S_i$ for some $n > 0$. On the other hand, we have $(m,N)
\in S$ for some $m \in \Z$. Since $n > 0$, there exists some $k > 0$
such that $nk + m \geq M_i$. It follows easily that
$$(M_i,N) = (M_i,nk)(m,N)(M_i-nk,-nk) = (M_i,n)^k(m,N)(M_i,n)^{-k} \in S.$$
Now $(M_i,n) \in S_i$ yields $(M_i,N) \in S_i$. Therefore 
$\{ (M_i,N)\} \cup S'_i \subseteq S_i$ and so 
$$\langle \{ (M_i,N)\} \cup S'_i \rangle \subseteq S_i.$$ 

Conversely, let $(r,s) \in S_i$. We may assume that $s \geq 0$. If $r
> M_i$, then $s = 0$ by maximality of $M_i$ and so $(r,s) \in
S'_i$. Thus we may assume that $r \leq M_i$. Since $(r,s),(M_i,N) \in
S_i$, we may write $(r,s) = (M_i -pN,qN)$ for some $p,q \geq
0$. Assuming that $(M_i,N)^0$ denotes an absent factor, it
follows that
$$\begin{array}{lll}
(r,s)&=&(M_i-pN,-pN)(M_i,(p+q)N)\\
&=&(M_i-pN,-pN)(M_i,(p+q+1)N)(M_i-N,-N)\\
&=&(M_i,N)^{-p}(M_i,N)^{p+q+1}(M_i,N)^{-1}
\end{array}$$
and so $S_i \subseteq \langle \{ (M_i,N)\} \cup S'_i
\rangle$. Thus (\ref{zz1}) holds. Since $S$ is bounded, $S'_i$ is
finite and so each nonempty $S_i$ is
finitely generated. Therefore (\ref{zz}) holds.

Finally, we show that $E \sdp G$ is a Howson inverse semigroup. Let
$S,S'$ be finitely generated inverse subsemigroups of $E\sdp G$. We may
assume that $S \cap S'$ is infinite, hence $S$ and $S'$ are both
infinite. Since finitely generated semilatices are finite, then both
$S$ and $S'$ contain nonidempotents, say $(m,n)$ and $(m',n')$,
respectively. Without loss of generality, we may assume that $n,n' >
0$. Hence $(m,nn') = (m,n)^{n'} \in S$ and $(m',nn') = (m',n')^n \in
S'$. 

Suppose that $S\cap S'$ contains only idempotents. By our previous
remarks on boundedness, both $S$ and $S'$ are bounded and so is $S\cap
S'$. Since $S\cap S'$ is infinite, it follows that $(r,0) \in S\cap S'$
for some $r \leq m,m'$. Hence
$$(r,nn') = (r,0)(m,nn') \in S, \quad (r,nn') = (r,0)(m',nn') \in
S',$$
and so $S \cap S'$ would contain a nonidempotent, a
contradiction. Therefore $S\cap S'$ must contain a
nonidempotent. Since $S \cap S'$ is bounded, it follows from
(\ref{zz}) that $S\cap S'$ is finitely generated. Therefore $E\sdp G$
is a Howson inverse semigroup.

\section*{Acknowledgements}

The first author acknowledges support from the European Regional
Development Fund through the programme COMPETE
and the Portuguese Government through FCT (Funda\c c\~ao para a Ci\^encia e a
Tecnologia) under the project PEst-C/MAT/UI0144/2013.
This work was developed within the activities of FCT's project PEst-OE/MAT/UI0143/2013-14
of the Centro de \'Algebra da Universidade de Lisboa (CAUL),
that supported the visit of the first author to CAUL in July 2014.


\begin{thebibliography}{99}

\bibitem{AS75} \textsc{A.~W.~Anisimov}, \textsc{F.~D.~Seifert}, {\em Zur
  algebraischen Charakteristik der durch kontext-freie Sprachen
  definierten Gruppen}, Elektron. Informationsverarbeit. Kybernetik,
  {\bf 11} no.\ 10--12 (1975), 695--702.

\bibitem{ASS14} \textsc{V.~Ara\'ujo}, \textsc{P.~V.~Silva},
  \textsc{M.~Sykiotis}, {\em Finiteness 
  results for subgroups of finite extensions}, J.~Algebra, {\bf 423}
  (2015), 592--614.

\bibitem{BS10} \textsc{L.~Bartholdi}, \textsc{P.~V.~Silva},
  {\em Rational subsets of groups},
  to appear in ``Handbook of Automata Theory'', ed. J.-E. Pin, Chapter~23,
  arXiv:1012.1532, 2010.

\bibitem{B79} \textsc{J.~Berstel}, ``Transductions and Context-Free Languages'',
  B.~G.~Teubner, 1979.

\bibitem{BGS14} \textsc{M.~Branco}, \textsc{G.~M.~S.~Gomes},
  \textsc{P.~V.~Silva}, {\em Takahasi semigroups}, preprint, 2014.

\bibitem{F14} \textsc{J.~Friedman},
 \emph{Sheaves on graphs, their homological invariants, and a proof of the
 Hanna Neumann Conjecture},
 Memoirs\ of\ the\ AMS, {\bf 233} (2014).

\bibitem{How95} \textsc{J.~M.~Howie}. ``Fundamentals of Semigroup
  Theory'', Clarendon Press, Oxford, 1995.

\bibitem{H54} \textsc{A.~G.~Howson},
 \emph{On the intersection of finitely generated free groups},
 J.\ London\ Math.\ Soc., {\bf 29} (1954), 428--434.

\bibitem{M13} \textsc{I.~Mineyev},
 \emph{Submultiplicativity and the Hanna Neumann Conjecture},
 Ann.\ of\ Math., {\bf 175}~no.1 (2012), 393--414.
 
\bibitem{N56} \textsc{H.~Neumann},
 \emph{On the intersection of finitely generated free groups},
 Publ.\ Math.\ Debrecen, {\bf 4} (1956), 186--169.
 \emph{On the intersection of finitely generated free groups. Addendum},
 Publ.\ Math.\ Debrecen, {\bf 5} (1957), 128.
 
\bibitem{JT89} \textsc{P.~R.~Jones}, \textsc{P.~G.~Trotter},
 \emph{The Howson property for free inverse semigroups},
 Simon Stevin, {\bf 63}  no.\ 3--4 (1989), 277--284.

\bibitem{Law99} \textsc{M.~V.~Lawson}, ``Inverse Semigroups: The Theory
  of Partial Symmetries'', World Scientific, 1999.  

\bibitem{OC76} \textsc{L.~O'Carroll},
 \emph{Embedding theorems for proper inverse semigroups},
 Journal of Algebra, {\bf 42} (1976), 26--40.


\bibitem{Sil91} \textsc{P.~V.~Silva}, ``Contributions to
  combinatorial semigroup theory'' (Ph.D. Thesis), University 
of Glasgow, 1991.

\bibitem{Sil93} \textsc{P.~V.~Silva}, {\em Normal-convex embeddings of
  inverse semigroups}, Glasgow Mathematical Journal {\bf 35} (1993), 115--121.
\end{thebibliography}
\end{document}